\documentclass{article}
\usepackage[utf8]{inputenc}
\usepackage{amsmath,amsthm}
\usepackage{amsfonts}
\usepackage{graphicx}
\usepackage{authblk}
\usepackage{biblatex}

\graphicspath{ {./images/} }

\newtheorem{theorem}{Theorem}[section]
\newtheorem{lemma}{Lemma}[section]
\newtheorem{corollary}{Corollary}[section]
\newtheorem{proposition}{Proposition}[section]

\theoremstyle{definition}
\newtheorem{definition}{Definition}[section]
\newtheorem{example}{Example}[section]
\newtheorem{remark}{Remark}[section]

\title{A discrete analog of Segre's theorem on spherical curves}
\author{Samuel Pacitti Gentil, Marcos Craizer}
\affil{\begin{small} Pontifícia Universidade Católica, Rio de Janeiro

sampacitti@gmail.com, craizer@puc-rio.br \end{small}}

\date{}

\begin{document}

\maketitle

\begin{abstract}
We prove a discrete analog of a certain four-vertex theorem for space curves. The smooth case goes back to the work of Beniamino Segre and states that a closed and smooth curve whose tangent indicatrix has no self-intersections admits at least four points at which its
torsion vanishes. Our approach uses the notion of discrete tangent indicatrix of a (closed) polygon. Our theorem then states that a polygon with at least four vertices and whose discrete tangent indicatrix has no self-intersections admits at least four flattenings, i.e., triples of vertices such that the preceding and following vertices are on the same side of the plane spanned by this triple.
\end{abstract}

\section{Introduction}

The Four-Vertex Theorem is a remarkable result in plane geometry. It states that for a smooth plane convex curve there are at least four points of the curve at which its curvature attains a maximum or minimum. Such points are called the \textbf{vertices} of the curve. In order to generalize this theorem for three-dimensional curves, it is necessary to reformulate not only the notion of convexity for such curves, but also the notion of a vertex.

One of the ways of capturing the notion of convexity of a plane curve is to look at the tangent vector at each point: a plane curve is convex if and only if, for each pair of distinct points of the curve, their respective tangent vectors do not point to the same direction. For space curves that satisfy this property, Segre proved that there are at least four points at which the torsion is zero (see \cite{Segre}). His theorem is actually a consequence of a theorem on the tangent indicatrix of the original curve: if a smooth spherical curve intersects every great circle (i.e., is not contained in any open hemisphere) and does not have self-intersections, then it has at least four spherical inflection points.

An alternative strategy to deal with this type of theorem is to consider the discrete case: instead of using smooth curves, the object of study consists of polygons. This approach simplifies considerably the problem, enables us to use induction on the number of the vertices, and makes it possible to use tools from combinatorics. The downside of this strategy is the ambiguity of the process of discretizing: there might be more than a way of doing so. Consequently, there might be discrete versions of theorems from the smooth case which are not equivalent to each other (i.e., a point of a polygon might or not be "a point of torsion zero" depending on the notion of "torsion zero" being used).

In the present text we will focus on the discrete case. Although a proof for the discrete case has already been found by Panina (see \cite{Panina}), her proof uses the original theorem by Segre (i.e., the smooth case) in order to prove the discrete analog. Our proof, on the other hand, relies solely on combinatorial arguments. For a generic polygon, we explain what it means for a triple of vertices to be a flattening/inflection and we define its discrete tangent indicatrix as a certain spherical polygon. Our Theorem then states that if this spherical indicatrix is not contained in any closed hemisphere and does not have self-intersections, then the original space polygon will admit at least 4 flattenings. The general strategy to prove this theorem is to use induction on the number of vertices of the spherical polygon. The most difficult and subtlest point of the induction step is to prove that there is at least one point that can be deleted from the spherical polygon so that the resulting spherical polygon still will neither be contained in any closed hemisphere nor will have self-intersections. In order to prove this fact, we also obtain some interesting results regarding spherical polygons in general using basic tools of Convex Geometry. At the end of the article we present two applications of the Main Theorem: a discrete version of the so called Tennis Ball Theorem and a discrete version of a theorem by Möbius.

\section{Basic Definitions}

We begin this section with two definitions:
\begin{definition}
\label{segre_curve}
A closed curve $\gamma: \mathbb{S}^1 \rightarrow \mathbb{R}^3$ is called a \textbf{Segre curve} if it has non-vanishing curvature and if, for any $t_1 \neq t_2 \in \mathbb{S}^1$, the tangent vectors $\gamma'(t_1)$ and $\gamma'(t_2)$ do not point to the same direction.
\end{definition}

\begin{definition}
A \textbf{flattening} or a \textbf{inflection point} is a point $p=\gamma(t_0)$ of a curve, for some $t_0 \in \mathbb{S}^1$, at which $\tau(t_0)=0$.
\end{definition}

\begin{theorem}
\label{segre_theorem}
\textbf{(Segre)} Any Segre curve has at least 4 flattenings.
\end{theorem}

Now, let $P = [v_1,v_2,...,v_n]$ be a \textbf{polygon}, i.e., a closed polygonal line (where we consider the indices $i$ modulo $n$). We say that $P$ is \textbf{generic} if it does not have 4 of its vertices on the same plane. A naive approach to discretizing Definition \ref{segre_curve} would be as follows: a ``Segre polygon" should not have directed edges pointing to the same direction. Notice, however, that any generic polygon satisfies this condition: if there were $e_i = \overrightarrow{v_i v_{i+1}}$ and $e_j = \overrightarrow{v_j v_{j+1}}$ with $e_i \parallel e_j$, then the vertices $v_i,v_{i+1},v_j$ and $v_{j+1}$ would be in the same plane, contradicting the genericity of the polygon.

Before we present a better approach for this problem, we need to introduce the following definition:

\begin{definition}
Given a smooth curve $\gamma$ in $\mathbb{R}^3$ with non-vanishing curvature, translate the unit tangent vector at each point of the curve to a fixed point \textbf{0}. The endpoints of the translated vectors describe then a curve on the unit sphere $\mathbb{S}^2$. We call this curve the \textbf{tangent indicatrix} of $\gamma$.
\end{definition}

Therefore, a Segre curve can be reformulated as a closed curve such that its tangent indicatrix is embedded in $\mathbb{S}^2$ (i.e., smooth and without self-intersections). Moreover, Theorem \ref{segre_theorem} now reads:

\begin{theorem}
Let $\gamma$ be a closed curve in $\mathbb{R}^3$. If its tangent indicatrix is embedded in $\mathbb{S}^2$, then $\gamma$ has at least 4 flattenings.
\end{theorem}

\begin{definition}
Given a polygon $P = [v_1,v_2,...,v_n]$ in $\mathbb{R}^3$, denote by $u_i$ the unit tangent vector with the same direction of the edge $e_i$, i.e.,
$$ u_i = \frac{e_i}{|e_i|} = \frac{\overrightarrow{v_i v_{i+1}}}{|\overrightarrow{v_i v_{i+1}}|} = \frac{v_{i+1}-v_i}{|v_{i+1}-v_i|}.$$
We define the \textbf{(discrete) tangent indicatrix} of $P$ as the \textit{closed spherical polygonal line}, i.e., the \textit{spherical polygon}
$$ Q = [u_1,u_2,...,u_n],$$
whose edges are the spherical segments (with minimal length) joining $u_i$ and $u_{i+1}$. This definition goes back to the work of Banchoff (see \cite{Banchoff}).
\end{definition}

We can finally define the discrete counterpart of Definition \ref{segre_curve}:

\begin{definition}
A polygon $P$ is a \textbf{Segre polygon} if its tangent indicatrix $Q$ does not have self-intersections.
\end{definition}

\begin{definition}
A \textbf{flattening} of a polygon is a triple $\{v_i,v_{i+1},v_{i+2}\}$ such that $v_{i-1}$ and $v_{i+3}$ are on the same side of the plane generated by vertices $v_i$, $v_{i+1}$ and $v_{i+1}$ (see figure \ref{figure_flattening_non_flattening}).
\end{definition}

\begin{figure}
    \centering
    \includegraphics[scale=0.3]{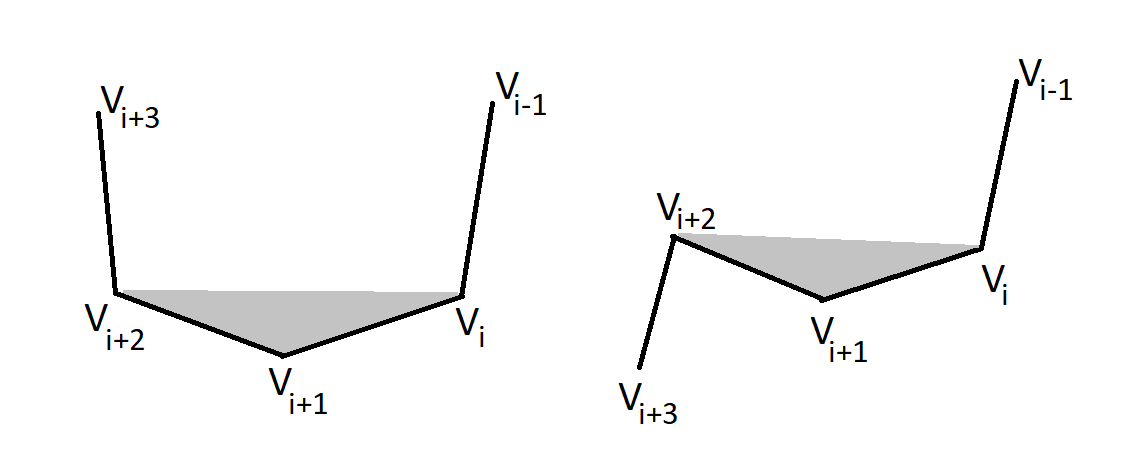}
    \caption{A flattening on the left, a non-flattening on the right}
    \label{figure_flattening_non_flattening}
\end{figure}

\begin{remark}
The previous definition implies that, if the triple $\{v_i,v_{i+1},v_{i+2}\}$ is a flattening, then the vectors $e_{i-1}$ and $e_{i+2}$ point to different sides of the plane generated by $\{v_i, e_i, e_{i+1}\}$. This in turn implies that $u_{i-1}$ and $u_{i+2}$ are on different sides of $span\{u_i,u_{i+1}\}$ (see figure \ref{figure_spherical_flattening}).
\end{remark}

\begin{figure}
    \centering
    \includegraphics[scale=0.3]{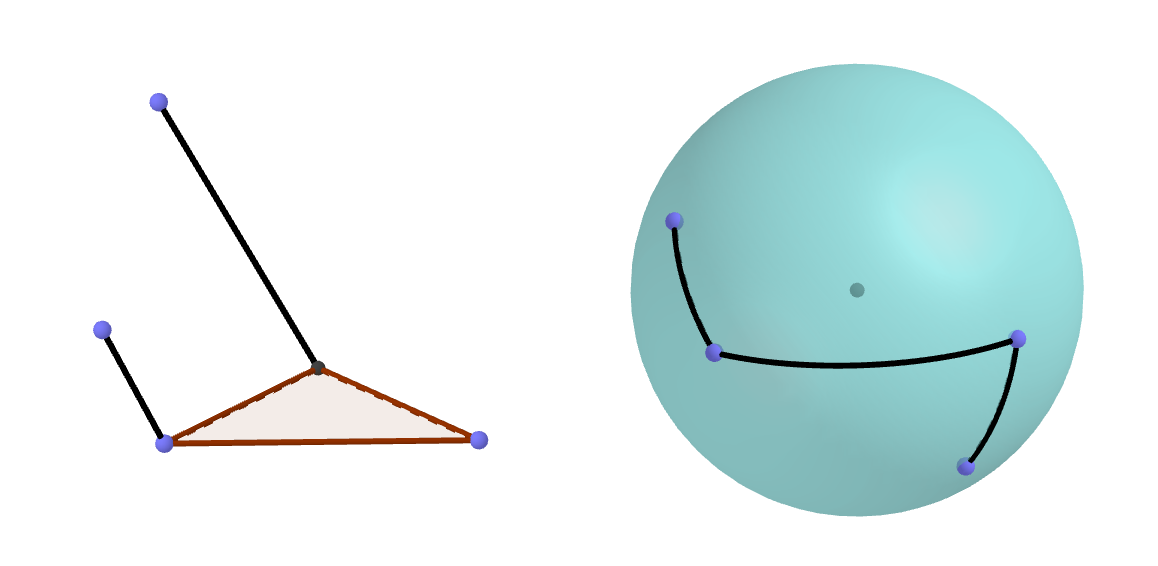}
    \caption{Flattening of a polygon $P$ and the corresponding inflection of the tangent indicatrix $Q$.}
    \label{figure_spherical_flattening}
\end{figure}

The main goal of this article is to prove the following result:

\begin{theorem}
\label{segre_theorem_discrete}
A Segre polygon with at least 4 vertices has at least 4 flattenings.
\end{theorem}

It is important to notice that, although this theorem states the result for Segre polygons in $\mathbb{R}^3$, its proof will work entirely within the realm of certain spherical polygons in $\mathbb{S}^2$. To get a feeling by what we mean by this, first we notice that the previous remark suggests the following definition:

\begin{definition}
Given a spherical polygon $Q \subset \mathbb{S}^2$, a \textbf{(spherical) inflection} of $Q$ is a pair $\{u_i,u_{i+1}\}$ such that $u_{i-1}$ and $u_{i+2}$ are in different sides of the plane spanned by $\{u_i,u_{i+1}\}$. Equivalently, $u_{i-1}$ and $u_{i+2}$ are in different hemispheres determined by the spherical line spanned by $\{u_i,u_{i+1}\}$.
\end{definition}

The condition that $u_{i-1}$ and $u_{i+2}$ are in different hemispheres determined by the spherical line spanned by $\{u_i,u_{i+1}\}$ is equivalent to the condition that the determinants $\epsilon_{i-1} = [u_{i-1},u_i,u_{i+1}]$ and $\epsilon_i = [u_i,u_{i+1},u_{i+2}] = [u_{i+2},u_i,u_{i+1}]$ have opposite signs. Consequently, Theorem \ref{segre_theorem_discrete} states that, if a spherical polygon $Q$ is the tangent indicatrix of a polygon and does not have self-intersections, then the cyclic sequence $(\epsilon_1,\epsilon_2,...,\epsilon_n$) has at least 4 sign changes (where each $\epsilon_i$ is defined as $[u_i,u_{i+1},u_{i+2}]$).

\section{The Cone Condition}

It is instructive to see first the following example of spherical polygon.

\begin{example}
\label{not_tangent_indicatrix}
Let $Q = [u_1,u_2,u_3,u_4] \subset \mathbb{S}^2$ be a spherical polygon, where
$u_1 = \left( \frac{\sqrt{2}}{2}, 0, \frac{\sqrt{2}}{2} \right)$,
$u_2 = \left( 0, \frac{\sqrt{2}}{2}, \frac{\sqrt{2}}{2} \right)$,
$u_3 = \left( -\frac{\sqrt{2}}{2}, 0, \frac{\sqrt{2}}{2} \right)$ and
$u_4 = \left( 0, -\frac{\sqrt{2}}{2}, \frac{\sqrt{2}}{2} \right)$.

It is not hard to (visually) see that $Q$ does not have any flattenings (see figure \ref{figure_counterexample}). It is also not hard to check algebraically that this is indeed the case: all determinants $\epsilon_1$, $\epsilon_2$, $\epsilon_3$ and $\epsilon_4$ are positive.
\end{example}

\begin{figure}
    \centering
    \includegraphics[scale=0.3]{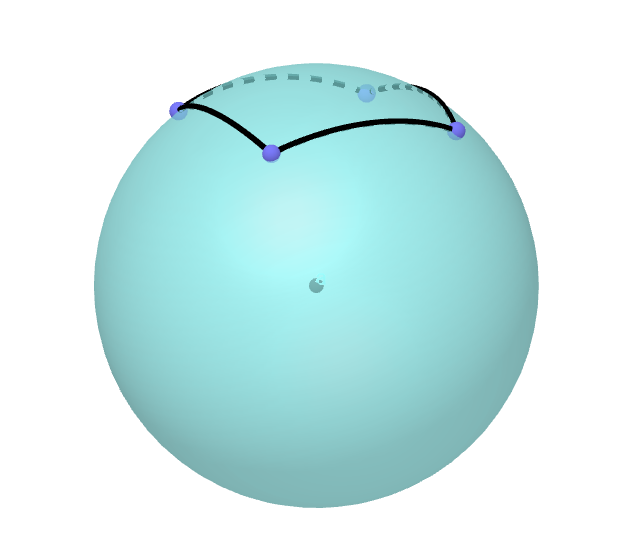}
    \caption{A spherical polygon without inflections.}
    \label{figure_counterexample}
\end{figure}

The spherical polygon of the previous example does not have any self-intersections, but it does not have any inflections. Is that a counterexample to our result? The answer is no. The reason why this happened is because \textbf{there is no space polygon whose tangent indicatrix is $Q$}.

More generally, consider a spherical polygon $Q \subset \mathbb{S}^2$ contained in a closed hemisphere, but not entirely contained in a spherical line. We may assume that this hemisphere is the one above the $xy$-plane (rotate the sphere $\mathbb{S}^2$ for that to be the case). This implies that all vectors $u_1,u_2,...,u_n$ (and consequently $e_1,e_2,...,e_n$) have $z$-coordinate equal or greater than zero. Since $Q$ is not entirely contained in any spherical line, at least one of the $e_i$, say $e_n$, have positive $z$-coordinate. Suppose that there is a space polygon $P = [v_1,v_2,...,v_n]$ whose tangent indicatrix is $Q$. Denoting by $z(v_i)$ and $z(e_i)$ the $z$-coordinate of $v_i$ and $e_i$ respectively, we have
$$z(v_1) \leq z(v_1) + z(e_1) = z(v_2) \leq z(v_2) + z(e_2) = z(v_3) \leq ...,$$
since each $z(e_i)$ is equal or greater than zero. Now, because $z(e_n)$ is strictly greater than zero, we have
$$ z(v_1) \leq z(v_2) \leq ... \leq z(v_{n-1}) \leq z(v_n) < z(v_n) + z(e_n) = z(v_1),$$
i.e., $z(v_1) < z(v_1)$. This contradiction implies that there is no such polygon $P$ whose tangent indicatrix is $Q$. We have therefore proved

\begin{proposition}
\label{necessity_tangent_indicatrix}
A necessary condition for a spherical polygon $Q \subset \mathbb{S}^2$, not entirely contained in a spherical line, to be the tangent indicatrix of some polygon $P \subset \mathbb{R}^3$ is that it cannot be contained in any closed hemisphere, i.e.,it must intersect every great circle of $\mathbb{S}^2$.
\end{proposition}

It turns out that the converse of the previous proposition is also true.

\begin{proposition}
\label{proposition_lifting}
If a spherical polygon $Q \subset \mathbb{S}^2$, not entirely contained in a spherical line, is not contained in any closed hemisphere (equivalently, it intersects every great circle of $\mathbb{S}^2$), then $Q$ is the tangent indicatrix of some polygon $P$ in $\mathbb{R}^3$.
\end{proposition}

A proof of Proposition \ref{proposition_lifting} will be provided later. In order to do so, as well to prepare the way for the proof of Theorem \ref{segre_theorem_discrete}, it will be again useful to express the geometry of the configuration of points in terms of determinants.

Recall what we have done so far: given a polygon $P = [v_1,...,v_n] \subset \mathbb{R}^3$, we calculated its edges $\{e_1,...,e_n\}$ and normalized them, obtaining $\{u_1,...,u_n\}$. Now we want to go the other way around: given $Q = [u_1,...u_n]$, we must obtain $\{e_1,...,e_n\}$ as an edge set of some polygon $P = [v_1,...,v_n]$.

Notice that, since a space polygon $P$ is closed,
$$ v_1 + e_1 + e_2 + ... + e_n = v_2 + e_2 + ... + e_n = ... =$$
$$ = v_n + e_n = v_1.$$
Thus $e_1 + ... + e_n = 0$, the zero vector. Conversely, if $e_1,...,e_n$ are such that their sum is zero, then one can choose an arbitrary point $v \in \mathbb{R}^3$ and put $v_1 = v$, $v_2 = v_1 + e_1$, ..., and $v_n = v_{n-1} + e_{n-1}$. Since $v_n + e_n = v_1 + e_1 + e_2 + ... + e_{n-1} + e_n = v_1 + 0 = v_1$, we obtain a closed polygon $P = [v_1,...,v_n]$ whose "not normalized tangent indicatrix" is the space polygon $[e_1,...,e_n]$.

Therefore, it is easy to pass from $\{e_1,...,e_n\}$ to $\{v_1,...,v_n\}$. The difficult step is, given $\{u_1,...,u_n\}$, to rescale them so that the new vectors sum up to zero. Since for each $i \in \{1,...,n\}$ the vectors $e_i$ and $u_i$ point to the same direction, what we want are positive real numbers $\lambda_1,...,\lambda_n$ such that $e_i = \lambda_i u_i$ for each $i \in \{1,...,n\}$, and with sum
$$ e_1 + ... + e_n = \lambda_1 u_1 + ... + \lambda_n u_n $$
equal to zero.

One can already see how we can use the fact of the $\{u_1,...,u_n\}$ not be entirely contained in one hemisphere: for any $u_i$ there must be a certain number of vectors which, for a convenient sum, cancel out the (possibly rescaled) vector $u_i$. At this point we introduce the following definitions:

\begin{definition}
Given m vectors $w_1,...,w_m \in \mathbb{R}^N$, the \textbf{closed cone} generated by $\{w_1,...,w_m\}$ is the set defined by
$$ \overline{\mathcal{C}}(w_1,...,w_n) = \{\lambda_1 w_1 + ... + \lambda_n w_m ; \lambda_i \geq 0 \text{ for each } i \in \{1,...,m\}\}.$$
Similarly, the \textbf{open cone} generated by $\{w_1,...,w_m\}$ is the set defined by
$$ \mathcal{C}(w_1,...,w_n) = \{\lambda_1 w_1 + ... + \lambda_n w_m ; \lambda_i > 0 \text{ for each } i \in \{1,...,m\}\}.$$
\end{definition}

Given 3 linearly independent vectors $u_2,u_3,u_4 \in \mathbb{R}^3$, any vector $u \in \mathbb{R}^3$ can be written as a unique linear combination of $u_2$, $u_3$ and $u_4$. If in addition $u$ is contained in $\mathcal{C}(u_2,u_3,u_4)$, then the coefficients $\lambda_2,\lambda_3,\lambda_4$ are all positive.  It is clear that in this case the plane $span\{u_2,u_3\}$ does not separate $u$ and $u_4$. In terms of determinants, this means that
$$ sign[u,u_2,u_3] = sign[u_4,u_2,u_3] = sign[u_2,u_3,u_4].$$
Analogously, we deduce that
$$ sign[u,u_3,u_4] = sign[u_2,u_3,u_4]$$
and
$$ sign[u,u_2,u_4] = sign[u_3,u_2,u_4] = -sign[u_2,u_3,u_4].$$

Conversely, it is clear that if a unit vector $u$ satisfies the above three equations, then $u \in \mathcal{C}(u_2,u_3,u_4)$.

Now, using the notion of cone, let us examine the following situation: Suppose that we are given four unit vectors $\{u_1,u_2,u_3,u_4\} \subset \mathbb{R}^3$, not all contained in the same closed hemisphere. Assume $\{u_2,u_3,u_4\}$ to be linearly independent.

We claim that $-u_1 \in C(u_2,u_3,u_4)$. For suppose that this is not the case, i.e., that one of the last three equations (say the first one) does not hold. If
$$ sign[-u_1,u_2,u_3] \neq sign[u_2,u_3,u_4],$$
then either $[u_1,u_2,u_3] = 0$ (in which case $u_1$, $u_2$ and $u_3$ are in the same spherical line and therefore $u_1$, $u_2$, $u_3$ and $u_4$ are on the same closed hemisphere) or $sign[u_1,u_2,u_3] = sign[u_2,u_3,u_4]$, i.e., $u_1$ and $u_4$ are on the same side of $span\{u_2,u_3\}$, i.e., $u_1$, $u_2$, $u_3$ and $u_4$ are on the same closed hemisphere.

Similarly, assuming that one of the other two equations does not hold, one gets another contradiction.

Now, assume that the set of unit vectors $\{u_1,u_2,u_3,u_4\}$ is contained in a closed hemisphere $\overline{H}$. Denote by $H'$ the open hemisphere which is the reflection of the open hemisphere $H$. Since $u_1 \in H$ and $C(-u_2,-u_3,-u_4) \cap \mathbb{S}^2 \subset H'$, and moreover $\overline{H} \cap H' = \emptyset$, it follows that $u_1 \notin C(-u_2,-u_3,-u_4)$.

We have therefore proved:

\begin{proposition}
\label{proposition_characterization}
Given any 4 vectors $u_1,u_2,u_3,u_4$ in $\mathbb{S}^2$ such that $\{u_2,u_3,u_4\}$ is linearly independent, the following conditions are equivalent:

(a) $u_1$, $u_2$, $u_3$ and $u_4$ are not on the same closed hemisphere;

(b) $-u_1 \in \mathcal{C}(u_2,u_3,u_4)$;

(c) $sign[u_1,u_2,u_3] = sign[u_1,u_3,u_4] = - sign[u_1,u_2,u_4] = - sign[u_2,u_3,u_4]$.
\end{proposition}

If $Q = [u_1,u_2,u_3,u_4]$ is a spherical polygon, then the hypothesis that the points are not on the same hemisphere is equivalent to $-u_1 \in \mathcal{C}(u_2,u_3,u_4)$, which in turn is equivalent to the fact that there are positive real numbers $\lambda_2$, $\lambda_2$ and $\lambda_3$ such that $-u_1 = \lambda_2 u_2 + \lambda_3 u_3 + \lambda_4 u_4$, i.e.,
$$ 1 \cdot u_1 + \lambda_2 u_2 + \lambda_3 u_3 + \lambda_4 u_4 = 0.$$

Therefore, in this case we were successful at lifting the vectors $u_1,...,u_4$
to rescaled vectors $e_1,...,e_4$ such that their sum is zero, which in turn implies the existence of (an infinite number of) polygons $P$ whose tangent indicatrix is $Q$.

An interesting and simple geometric fact that follows immediately from the previous proposition is the following:

\begin{corollary}
\label{corollary_characterization}
Given any 4 vectors $u_1,u_2,u_3,u_4$ in $\mathbb{S}^2$ such that any triple of them is linearly independent, the following conditions are equivalent:

(a) $-u_1 \in \mathcal{C}(u_2,u_3,u_4)$;

(b) $-u_2 \in \mathcal{C}(u_1,u_3,u_4)$.

(c) $-u_3 \in \mathcal{C}(u_1,u_2,u_4)$.

(d) $-u_4 \in \mathcal{C}(u_1,u_2,u_4)$.
\end{corollary}
\begin{proof}
By Proposition \ref{proposition_characterization}, all of the above conditions are equivalent to the condition that $u_1$, $u_2$, $u_3$ and $u_4$ are not on the same closed hemisphere.
\end{proof}

Now we want to look at configurations with more than just 4 points in $\mathbb{S}^2$. To have an idea of what problems might arise, let us look at the following examples:

\begin{example}
\label{example_three_cases}
Let $\{u_1,u_2,u_3,u_4,u_5\}$ be a set of 5 points in the sphere $\mathbb{S}^2$, where
$u_2 = \left( \frac{\sqrt{2}}{2}, 0, \frac{\sqrt{2}}{2} \right)$,
$u_3 = \left( 0, \frac{\sqrt{2}}{2}, \frac{\sqrt{2}}{2} \right)$,
$u_4 = \left( -\frac{\sqrt{2}}{2}, 0, \frac{\sqrt{2}}{2} \right)$ and
$u_5 = \left( 0, -\frac{\sqrt{2}}{2}, \frac{\sqrt{2}}{2} \right)$ (the same points of Example \ref{not_tangent_indicatrix}, except that the indices are translated by 1).

Depending on the position of the vector $u_1$, its antipode $-u_1$ might be in different regions of $\mathbb{S}^2$. Figure \ref{figure_cases_five_points} shows some of the possibilities.

In case (a), $-u_1 \in \mathcal{C}(u_2,u_3,u_4) \cap \mathcal{C}(u_2,u_3,u_5)$.

In case (b), $-u_1 \in \mathcal{C}(u_2,u_3,u_4) \cap \mathcal{C}(u_3,u_5)$.

Finally, in case (c), $-u_1 \in \mathcal{C}(u_2,u_4) \cap \mathcal{C}(u_3,u_5)$.

If $-u_1$ is not in one of these configurations, then $-u_1 \notin \mathcal{C}(u_2,u_3,u_4,u_5)$. Similarly as in the proof of Proposition \ref{proposition_characterization}, there are two possibilities:
\begin{itemize}
    \item $-u_1$ is on one of the spherical edges of the quadrangular region (say, $[u_2,u_3]$) and then, since $u_4$ and $u_5$ are on the same side of $span\{u_2,u_3\}$, we have that all of the points are one the same closed hemisphere.
    \item $-u_1$ is separated by the plane $span\{u_i,u_j\}$ (where $i,j$ are some indices of $\{2,3,4,5\}$) from the remaining pair $\{u_k,u_l\}$, i.e., the points $u_1$, $u_k$ and $u_l$ are on the same side of the plane $span\{u_i,u_j\}$. In other words, all points $u_1,...,u_5$ are on the same closed hemisphere.
\end{itemize}

Therefore, for a set of points not entirely contained in a hemisphere, the three cases above are (up to symmetry) the only possibilities. Thus:
\begin{itemize}
    \item In (a), $-u_1 = \lambda_2 u_2 + \lambda_3 u_3 + \lambda_4 u_4$ and $-u_1 = \mu_2 u_2 + \mu_3 u_3 + \mu_4 u_5$, which implies
    $$ 2 u_1 + (\lambda_2+\mu_2) u_2 + (\lambda_3 + \mu_3) u_3 + \lambda_4 u_4 + \mu_5 u_5 = 0.$$
    \item In (b), $-u_1 = \lambda_2 u_2 + \lambda_3 u_3 + \lambda_4 u_4$ and $-u_1 = \mu_3 u_3 + \mu_5 u_5$, which implies
    $$ 2 u_1 + \lambda_2 u_2 + (\lambda_3+\mu_3) u_3 + \lambda_4 u_4 + \mu_5 u_5 = 0.$$
    \item In (c), $-u_1 = \lambda_2 u_2 + \lambda_4 u_4$ and $-u_1 = \mu_3 u_3 + \mu_5 u_5$, which implies
    $$ 2 u_1 + \lambda_2 u_2 + \mu_3 u_3 + \lambda_4 u_4 + \mu_5 u_5 = 0.$$
\end{itemize}
In any of these three cases, we succeeded at rescaling our original unit vectors so that their new sum equals zero. Now, if these points were originally the vertices of a spherical polygon $Q = [u_1,u_2,u_3,u_4,u_5]$, this implies the existence of (a infinite number of) polygons whose tangent indicatrix is exactly $Q$.
\end{example}

\begin{figure}
    \centering
    \includegraphics[scale=0.4]{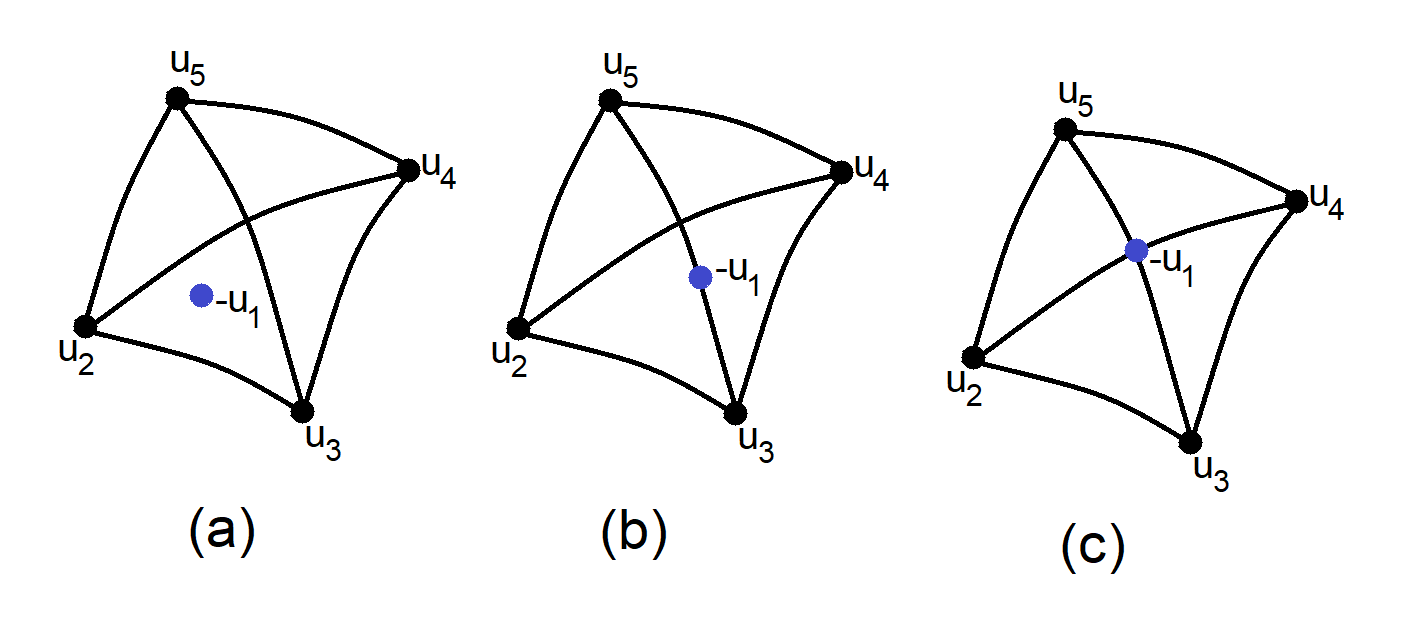}
    \caption{Three different cases}
    \label{figure_cases_five_points}
\end{figure}

\begin{example}
Now let us look at another configuration, as shown in figure \ref{figure_cases_five_points_degenerate}.

In case (a), $-u_1 \in \mathcal{C}(u_2,u_3,u_4) \cap \mathcal{C}(u_3,u_4,u_5)$, i.e., $-u_1 = \lambda_2 u_2 + \lambda_3 u_3 + \lambda_4 u_4$ and $-u_1 = \mu_3 u_3 + \mu_4 u_4 + \mu_5 u_5$ for positive $\lambda$'s and $\mu$'s, which implies that
$$ 2 u_1 + \lambda_2 u_2 + (\lambda_3+\mu_3) u_3 + (\lambda_4 + \mu_4) u_4 + \mu_5 u_5 = 0,$$ 
where all coefficients are positive.

In case (b), $-u_1 \in \mathcal{C}(u_2,u_3,u_4) \cap \mathcal{C}(u_3,u_5)$, i.e., $-u_1 = \lambda_2 u_2 + \lambda_3 u_3 + \lambda_4 u_4$ and $-u_1 = \mu_3 u_3 + \mu_5 u_5$ for positive $\lambda$'s and $\mu$'s, which implies that
$$ 2 u_1 + \lambda_2 u_2 + (\lambda_3+\mu_3) u_3 + \lambda_4 u_4 + \mu_5 u_5 = 0,$$ 
where all coefficients are positive.
These two cases are (up to symmetry) the only possibilities (if $-u_1 \notin \mathcal{C}(u_2,u_3,u_4,u_5)$, then we derive a contradiction in the same way as we did in Example \ref{example_three_cases}). Since we could rescale these points so that they sum to zero, we can then find a space polygon $P$ whose tangent indicatrix is $Q = [u_1,u_2,u_3,u_4,u_5]$.
\end{example}

\begin{figure}
    \centering
    \includegraphics[scale=0.35]{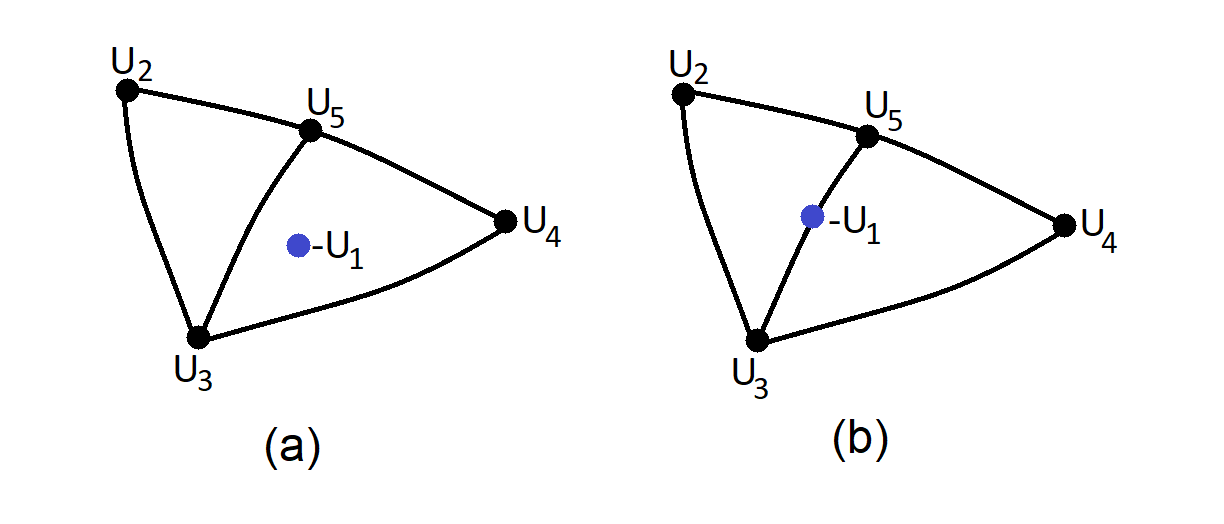}
    \caption{Two degenerate cases}
    \label{figure_cases_five_points_degenerate}
\end{figure}

\begin{example}
A third type of configuration is given by figure \ref{figure_cases_five_points_inside}. 

In case (a) and (b), $\-u_1 \in \mathcal{C}(u_2,u_3,u_4) \cap \mathcal{C}(u_2,u_3,u_5)$.

In case (c), $-u_1 \in \mathcal{C}(u_2,u_3,u_4) \cap \mathcal{C}(u_4,u_5)$.

In case (d), $-u_1 \in \mathcal{C}(u_2,u_3,u_4) \cap \mathcal{C}(u_5)$ Notice that in this case $u_5$ is the antipode of $u_1$.

Proceeding the same way as it was done in the previous examples, one shows that, if the $u_i$'s are not entirely contained in a hemisphere, then these four cases are (up to symmetry) the only possibilities. For each case one can then obtain rescaled versions of the $u_i$'s so that they sum up to zero. Therefore there is a space polygon $P$ whose tangent indicatrix is $Q = [u_1,u_2,u_3,u_4,u_5]$.
\end{example}

\begin{figure}
    \centering
    \includegraphics[scale=0.3]{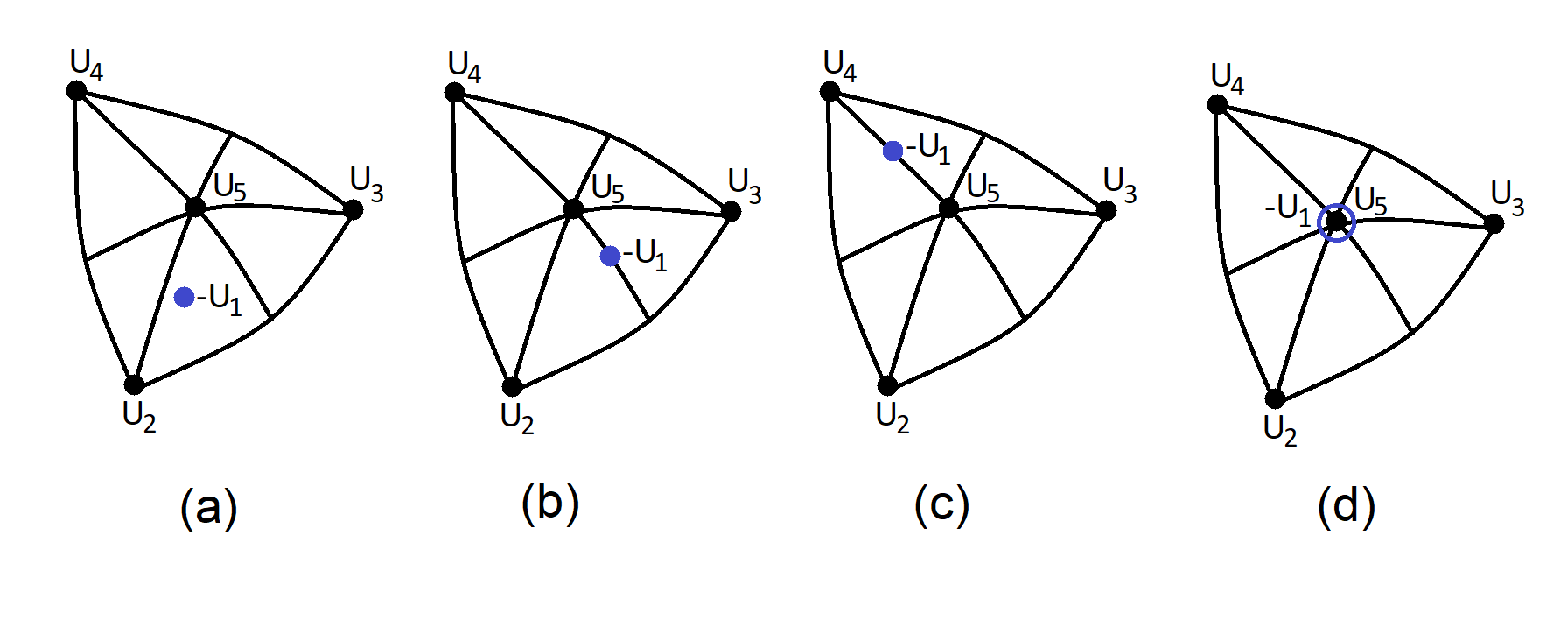}
    \caption{Four more cases}
    \label{figure_cases_five_points_inside}
\end{figure}

As the three previous examples have shown, a certain configuration of points determines a couple of cases to consider. A little thought might convince the reader that these three examples exhaust all possibilities for the relative position of the points $u_2$, $u_3$, $u_4$ and $u_5$ up to some permutation of the indices (the case in which $u_2$, $u_3$, $u_4$ and $u_5$ are in the same spherical line does not appear since it would then imply that all points from $u_1$ to $u_5$ would be on the same closed hemisphere).

Another important feature of what we have just done is that, given a unit vector $u_1$, every other unit vector from $u_2$, $u_3$, $u_4$ and $u_5$ appeared at least once as the generator of one of the cones that contained the antipode of $u_1$ (some of them appeared more than once, but that does not matter). This made it possible to obtain a sum with all the (rescaled) vectors.

Moreover, as we have seen in these examples, there might happen that three different points are in the same spherical line (i.e.,three vectors on the same plane). Although we could prove Proposition \ref{proposition_lifting} in this more general case, it will be convenient to assume that this does not happen. The reason is twofold: it will make the proof considerably simpler and, as we will see later, any spherical polygon with three non-consecutive collinear vertices can be perturbed into a spherical polygon with no three collinear vertices, but with the same number of inflections.

From now on, \textbf{we assume the following typographical conventions}: given vectors $u_i$,$u_j$ and $u_k$, we may also write $[i,j,k]$ instead of $[u_i,u_j,u_k]$. Additionally, the notation
$$ [i,j,k] \simeq [a,b,c] $$
means that $sign[i,j,k] = sign[a,b,c]$. Therefore, if the determinants have opposite signs, we could write
$$ [i,j,k] \simeq -[a,b,c].$$

Since \textbf{we are also assuming from now on that the spherical polygons considered do not have three points in the same spherical line}, any determinant calculated using a triple of the points of the spherical polygon is nonzero. Thus, in this case, $[i,j,k] \simeq -[a,b,c]$ is equivalent to and will be written as $[i,j,k] \not \simeq [a,b,c]$.

\section{Some results of Convex Geometry}

Before proving Proposition \ref{proposition_lifting}, we need some results of Convex Geometry. The proof of the first one reveals an interplay between conical sets and the notion of \textit{convexity in the sphere}. For an account of these ideas, the interested reader might consult \cite{Ferreira}.

\begin{lemma}
\label{lemma_all_R3}
Let $Q = \{u_1, u_2,..., u_n \}$ be a finite set of points on the sphere $\mathbb{S}^2$ ($n \geq 4$), not all of them on the same closed hemisphere. Then $\overline{\mathcal{C}}(u_1,...,u_n) = \mathbb{R}^3$.
\end{lemma}

\begin{proof}
The proof is by induction on the number of points. For $n=4$, then Proposition \ref{proposition_characterization} (or equivalently, Corollary \ref{corollary_characterization}) implies that each point $u_i$ ($i=1,2,3,4$) is such that its antipode in the open cone spanned by the other points. Since the four closed cones (each one generated by a different triple of points from the set $\{u_1,u_2,u_3,u_4\}$), restricted to the sphere, divide it into four regions, we have that $\overline{\mathcal{C}}(u_1,u_2,u_3,u_4) = \mathbb{R}^3$.

Now, assume the result true for $n$, and suppose we are given a set of $n+1$ points $Q = \{u_1,...,u_n,u_{n+1}\}$, not all of them in the same closed hemisphere. If the set $Q - \{u_{n+1}\}$ is not in the same hemisphere, then by the induction hypothesis $\mathbb{R}^3 = \overline{\mathcal{C}}(u_1,...,u_n) \subset \overline{\mathcal{C}}(u_1,...,u_n,u_{n+1}) \subset \mathbb{R}^3$, from which the result follows.

If, however, $\{u_1,...,u_n\}$ is in some closed hemisphere $H$, consider then the open region $R = \mathcal{C}(u_1,...,u_n) \cap \mathbb{S}^2 \subset \overline{\mathcal{C}}(u_1,...,u_n) \cap \mathbb{S}^2 \subset H$. We may assume that the vertices of the topological boundary of this region are all the $u_i$'s of $Q$. For if it were not the case (say $u_j$ is the topological interior of $R$), then $\overline{\mathcal{C}}(u_1,...,u_n) = \overline{\mathcal{C}}(u_1,...,\hat{u}_j,...,u_n)$ and, consequently, $\overline{\mathcal{C}}(u_1,...,u_{n+1}) = \overline{\mathcal{C}}(u_1,...,\hat{u}_j,...,u_{n+1})$. By the induction hypothesis applied to the set $\{u_1,...,\hat{u}_j,...,u_{n+1}\}$,
$$ \overline{\mathcal{C}}(u_1,...,u_{n+1}) = \overline{\mathcal{C}}(u_1,...,\hat{u}_j,...,u_{n+1}) = \mathbb{R}^3.$$

After labelling the indices, if necessary, we may assume that the boundary of the region $R$ is a convex polygon with the ordering $u_1,u_2,...,u_n$ and oriented so that the $R$ is always on the left of the polygon. This region is, therefore, the intersection of all the open hemispheres
$$H_i = \{u \in \mathbb{S}^2 ; [u,u_i,u_{i+1}] > 0 \},$$
for $i=1,2,...,n$. Notice that $u_j \in H_i$, for all $j \neq i,i+1$.

We claim that $-u_{n+1} \in R$. For if it were not in $R$, then $-u_{n+1}$ would not be in at least one of the $H_i$. This would imply that $[-u_{n+1},u_i,u_{i+1}] \leq 0$, i.e., $[u_{n+1},u_i,u_{i+1}] \geq 0$, i.e., $u_{n+1} \in \overline{H}_i$. Since all other points $u_j$'s are in $\overline{H}_i$, this implies that all points are in the closed hemisphere $\overline{H}_i$, contrary to the hypothesis.

Now, since $u_{n+1}$ is not in any of the $H_i$, then the open cones of the form $\mathcal{C}(u_i,u_{i+1},u_{n+1})$ (for all $i = 1,...,n$) are disjoint and do not intersect the open cone $\mathcal{C}(u_1,...,u_n)$. Moreover, we have that
$$ \overline{\mathcal{C}}(u_1,...,u_n) \cup \bigcup_{i=1}^n \overline{\mathcal{C}}(u_i,u_{i+1},u_{n+1}) = \mathbb{R}^3.$$
Since the set on the left is contained in $\overline{\mathcal{C}}(u_1,...,u_n,u_{n+1}) \subset \mathbb{R}^3$, the result follows.
\end{proof}

The following proposition is the conical version of the known Carathéodory's Theorem for convex sets (see for instance \cite{Hug} or \cite{Pak}). The proof of the former is similar to the usual proof of the latter result.

\begin{proposition}
\label{proposition_caratheodory}
Let $Q$ be a finite set of $n$ points in $\mathbb{R}^d$ ($n \geq d$). Let $u$ be any point of $$\overline{\mathcal{C}}(Q) = \{ \text{ finite sums of elements of the form } \lambda u; \lambda \geq 0, u \in Q \}.$$
Then there are $d$ points $u_1,...,u_d$ in $Q$ and non-negative numbers $\lambda_1,...,\lambda_d$ such that
$$ u = \lambda_1 u_1 + ... + \lambda_d u_d.$$
\end{proposition}
\begin{proof}
Given $u \in \overline{\mathcal{C}}(Q)$, we have that
$$ u = \lambda_1 u_1 + ... + \lambda_m u_m,$$
with $\lambda_i \geq 0$, for all $i \in \{1,...,m\}$. Let $m$ be the minimal number for which such a conical combination for $u$ is possible.

We claim that $\{u_1,...,u_d\}$ is linearly independent (from which it follows that $m \leq d$). For if it were linearly independent, then there would be $\alpha_1,...,\alpha_m$, not all zero, such that
$$ \sum_{i=1}^m \alpha_i u_i = 0.$$
Let $I:= \{ i \in \{1,...,m\}; \alpha_i > 0 \}$ (which can be assumed to be nonempty, otherwise we could work with $-\alpha_i$'s instead of $\alpha_i$'s). Choose $i_0 \in I$ such that
$$ \frac{\lambda_{i_0}}{\alpha_{i_0}} = \min_{i\in I} \frac{\lambda_i}{\alpha_i}.$$
Hence,
$$ \lambda_i - \frac{\lambda_{i_0}}{\alpha_{i_0}} \alpha_i \geq 0,$$
for all $i \in I$ (notice also that this inequality always holds when $\alpha_i \leq 0$). Then, we have
$$ \sum_{i=1}^m \left( \lambda_i - \frac{\lambda_{i_0}}{\alpha_{i_0}} \alpha_i \right) u_i = \sum_{i=1}^m \lambda_i u_i - \frac{\lambda_{i_0}}{\alpha_{i_0}} \sum_{i=1}^m \alpha_i u_i = \sum_{i=1}^m \lambda_i u_i - \frac{\lambda_{i_0}}{\alpha_{i_0}} \cdot 0 = u,$$
with $ \lambda_i - \frac{\lambda_{i_0}}{\alpha_{i_0}} \alpha_i \geq 0 $, for all $i\in\{1,...,m\}$, and $\lambda_{i_0} - \frac{\lambda_{i_0}}{\alpha_{i_0}} \alpha_{i_0} = 0$. This contradicts minimality of $m$.
\end{proof}

\begin{lemma}
\label{lemma_pick_up}
Let $Q = \{u_1, u_2,..., u_n \}$ be a finite set of points on the sphere $\mathbb{S}^2$, with $n \geq 5$, not all of them on the same hemisphere. Then the set
$$ X = \{u_i \in Q; \{u_1,...,\hat{u}_i,...,u_n\} \text{ is not contained on a hemisphere} \}$$
has at least $n-3$ elements.
\end{lemma}

\begin{proof}
By Lemma \ref{lemma_all_R3}, $\overline{\mathcal{C}}(u_1,...,u_n) = \mathbb{R}^3$. In particular, $-u_1 \in \overline{\mathcal{C}}(u_1,...,u_n)$, i.e., there are non-negative numbers $\lambda_i$ ($1 \leq i \leq n$) such that
$$ -u_1 = \lambda_1 u_1 + \lambda_2 u_2 + ... + \lambda_n u_n,$$
i.e.,
$$ -u_1 = \mu_2 u_2 + ... + \mu_n u_n,$$
where $\mu_i = \lambda_i / (1+\lambda_1) \geq 0$. In other words, $-u_1 \in \overline{\mathcal{C}}(u_2,...,u_n)$. By Proposition \ref{proposition_caratheodory}, there are $u_i,u_j,u_k \in Q$ and non-negative numbers $\alpha_i$, $\alpha_j$ and $\alpha_k$ such that
$$ -u_1 = \alpha_i u_i + \alpha_j u_j + \alpha_k u_k.$$
Since we are assuming that there are no three spherically collinear points in $Q$, all three numbers $\alpha_i$, $\alpha_j$ and $\alpha_k$ are positive, i.e., $-u_1 \in \mathcal{C}(u_i,u_j,u_k)$.

After a relabelling of the indices, if necessary, we may assume that the points $u_i$, $u_j$ and $u_k$ are $u_2$, $u_3$ and $u_4$.

By Proposition \ref{proposition_characterization}, each one of the points $u_1$, $u_2$, $u_3$ and $u_4$ is such that its antipode is on the open cone spanned by the other points. The respective four closed cones divide the sphere into four regions. Since we are assuming that no three points of $Q$ are collinear, we have that any of the remaining points $u_5,...,u_n$ is such that its antipode is contained in one and only one of the cones $C(u_1,u_2,u_3)$, $C(u_1,u_2,u_4)$, $C(u_1,u_3,u_4)$ and $C(u_2,u_3,u_4)$.

Since the antipode of any point from $u_1$ to $u_n$ is in an open cone spanned by a triple from the four points $u_1$, $u_2$, $u_3$ or $u_4$, this means that all points from $u_5$ to $u_n$ are not essential as a cone generator. Hence all points from $u_5$ to $u_n$ are in the set $X$, as defined before.

Now, we just need to show that at least one of the points $u_1$, $u_2$, $u_3$ or $u_4$ is in $X$. Since $n \geq 5$ and no three points are (spherically) collinear, at least one of the following sets is non-empty:
$$ Q_1 = \{ u_i \in Q - \{u_1\}; -u_i \in C(u_2,u_3,u_4)\},$$
$$ Q_2 = \{ u_i \in Q - \{u_2\}; -u_i \in C(u_1,u_3,u_4)\},$$
$$ Q_3 = \{ u_i \in Q - \{u_3\}; -u_i \in C(u_1,u_2,u_4)\},$$
$$ Q_4 = \{ u_i \in Q - \{u_4\}; -u_i \in C(u_1,u_2,u_3)\}.$$
We may assume that this non-empty set is $Q_1$. This implies, by Proposition \ref{proposition_characterization}, that for some $i\in\{5,...,n\}$ the points $u_2$, $u_3$, $u_4$ and $u_i$ are not on the same hemisphere. This in turn implies that any of the remaining points (including $u_1$) is in one and only one of the four open cones spanned by each possible triple from $\{u_2,u_3,u_4,u_i\}$. Thus $u_1$ is not essential as a cone generator, i.e., $u_1 \in X$.
\end{proof}

\begin{remark}
The proof of Lemma \ref{lemma_pick_up} actually showed a stronger result: We could get rid of \textit{all points except four} at once so that the new configuration would still not be contained in a hemisphere.
\end{remark}

\begin{proof}
(of Proposition \ref{proposition_lifting}) Given a spherical polygon $Q = [u_1,u_2,...,u_n]$, we just have to show that there are positive scalars $\alpha_i$ such that the rescaled vectors $e_i = \alpha_i u_i$ sum up to zero.

The proof is on induction on the number of points $n \geq 4$. The case $n=4$ is Proposition \ref{proposition_characterization}: $-u_1 \in C(u_2,u_3,u_4)$, which implies that $1 \cdot u_1 + \alpha_2 \cdot u_2 + \alpha_3 \cdot u_3 + \alpha_4 \cdot u_4 = 0$.

Now, assume the result for $n$ points. Suppose we are given $(n+1)$ points, not all of them on the same hemisphere. By Lemma \ref{lemma_pick_up}, there is at least one point (say, $u_{n+1})$ such that the remaining points are not on the same hemisphere. By the induction hypothesis, there are positive $\lambda_i$ such that
$$ \lambda_1 \cdot u_1 + \lambda_2 \cdot u_2 + ... + \lambda_n \cdot u_n = 0.$$

By the proof of Lemma \ref{lemma_pick_up}, there are four points $u_i$, $u_j$, $u_k$ and $u_l$ (which can be assumed to be different from $u_{n+1}$) such that the four different cones divide the sphere in four regions. Since we assume that no three points of $Q$ are (spherically) collinear, we have that $-u_{n+1}$ is in one of these open cones, say $C(u_i,u_j,u_k)$. Therefore
$$ \mu_i \cdot u_i + \mu_j \cdot u_j + \mu_k \cdot u_k + 1 \cdot u_{n+1} = 0,$$
which, summing to the previous sum, gives
$$ \sum_{m=1}^{n+1} \alpha_m u_m = 0,$$
where $\alpha_m = \lambda_m + \mu_m$ for $m=i,j$ or $k$, $\alpha_{n+1}=1$ and $\alpha_m = \lambda_m$ for the remaining points.
\end{proof}

The reader might be wondering why we bothered to prove Lemma \ref{lemma_pick_up}, which is considerably stronger than what we actually used in the proof of Proposition \ref{proposition_lifting}. The reason why we need this result will become clear in the course of the proof of Theorem \ref{segre_theorem_discrete}.

In order to simplify language, we introduce the following terminology:

\begin{definition}
\label{definition_balanced}
A set of points $Q = \{u_1,...,u_n\} \subset \mathbb{S}^2 $ ($n \geq 4$), not in the same spherical line, is said to be \textbf{balanced} or in \textbf{balanced position} if its points are not in the same closed hemisphere. A point $u_i$ of a balanced set is said to be \textbf{essential} if the set $\{u_1,...,\hat{u}_i,...,u_n\}$ is not balanced. Otherwise $u_i$ is \textbf{nonessential}.
For a spherical polygon $Q = [u_1,..,u_n]$, the same definitions apply to $Q$ considered as a set of vertices.
\end{definition}

The condition of having all unit vectors $u_1,u_2,...,u_n$ in balanced position simply means that, for each $i \in \{1,...,n\}$, there is at least one triple of points $u_j$, $u_k$ and $u_l$ such that $-u_i \in \mathcal{C}(u_j,u_k,u_l)$. By Proposition \ref{proposition_characterization}, this is equivalent to
$$ [i,j,k] \simeq [i,k,l] \not \simeq [i,j,l] \simeq [j,k,l].$$

In Lemma \ref{lemma_pick_up}, however, we improved this even more: there are actually four specific points $u_i$, $u_j$, $u_k$ and $u_l$ such that any $u_m$ of the remaining points has its antipode located in one and only one of the four cones spanned by these points.

The next step is, therefore, to express the fact of a spherical polygon \textbf{not having self-intersection} as a relation of signs of determinants. Looking at figure \ref{figure_relative_position_edges} we have some possibilities regarding the relative position of two spherical edges. For the sake of simplicity of notation we assume that one (spherical) edge is $\overrightarrow{u_1 u_2}$ and the other is $\overrightarrow{u_5 u_6}$.

If edges $\overrightarrow{u_1 u_2}$ and $\overrightarrow{u_5 u_6}$ intersect, then the spherical line spanned by an edge separates the two endpoints of the other edge (see figure \ref{figure_relative_position_edges}(d)). In terms of determinants, this means that
$$ [1,2,5] \not \simeq [1,2,6] \text{ and } [1,5,6] \not \simeq [2,5,6].$$

\begin{figure}
    \centering
    \includegraphics[scale=0.2]{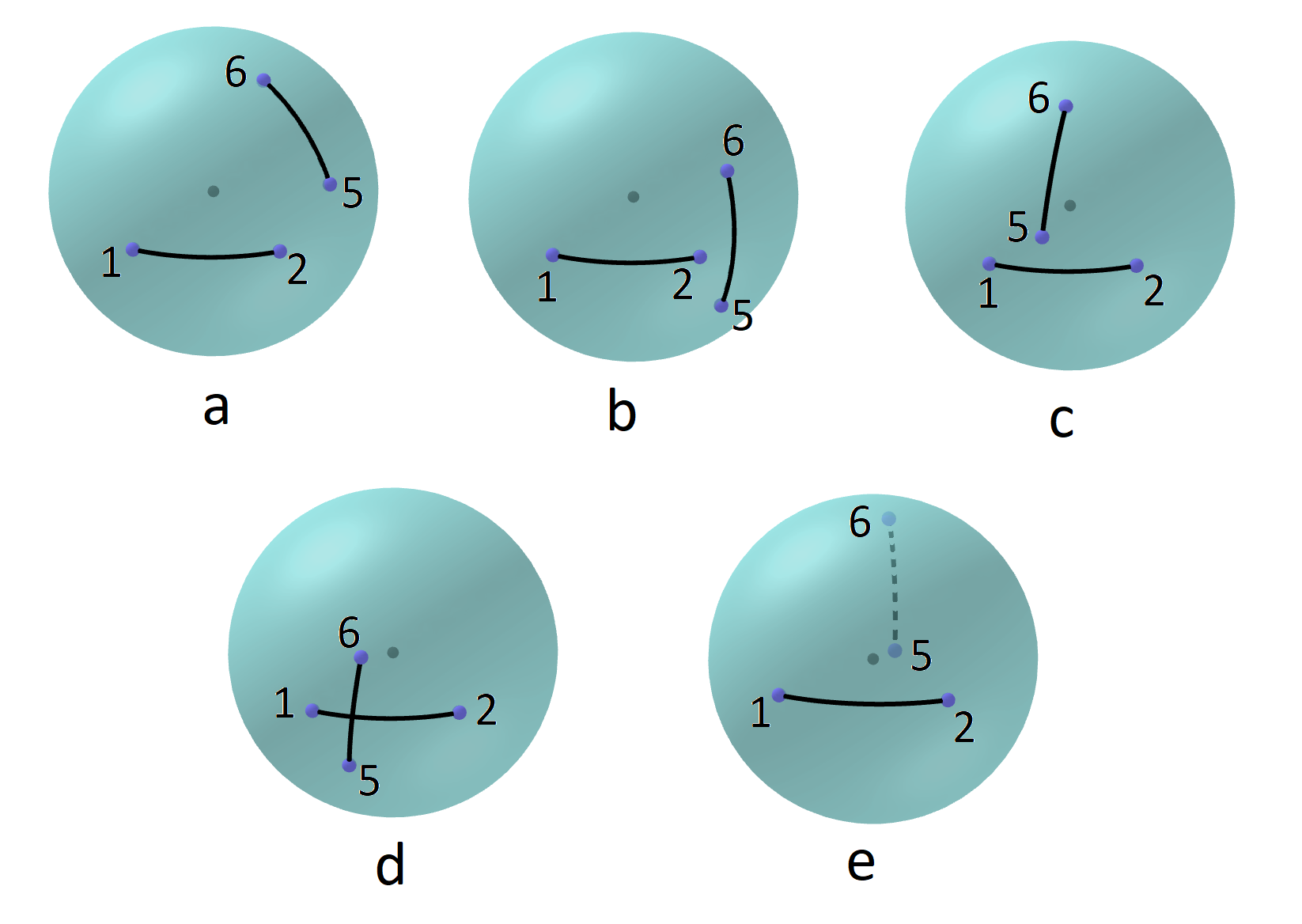}
    \caption{Possibilities regarding the relative positions of edges}
    \label{figure_relative_position_edges}
\end{figure}

Notice, however, that this relation is not exclusive to the case where both edges intersect: this relation is also true if one edge intersects the antipode of the other edge (see figure \ref{figure_relative_position_edges}(e)). In order to distinguish these two possibilities, notice that, if the edges intersect, then the spherical line spanned by $u_1$ and $u_6$ does not separate $u_2$ and $u_5$, while the spherical line spanned by $u_2$ and $u_5$ does not separate $u_1$ and $u_6$ (that would not be case if one edge intersected the antipode of the other edge).
Hence
$$ [1,6,2] \simeq [1,6,5] \text{ and } [2,5,1] \simeq [2,5,6].$$
i.e.,
$$ [1,2,6] \simeq [1,5,6] \text{ and } [1,2,5] \simeq [2,5,6].$$
Therefore, if edges $\overrightarrow{u_1 u_2}$ and $\overrightarrow{u_5 u_6}$ intersect, we have that
$$ [1,2,5] \simeq [2,5,6] \not \simeq [1,2,6] \simeq [1,5,6].$$

We have therefore proved

\begin{proposition}
\label{self_intersection_determinants}
A spherical polygon $Q \subset \mathbb{S}^2$ has a self-intersection at edges $\overrightarrow{u_i u_{i+1}}$ and $\overrightarrow{u_j u_{j+1}}$ (where $j \neq i+1$ and $i \neq j+1$) if and only if the relation
$$ [i,i+1,j] \simeq [i+1,j,j+1] \not \simeq [i,i+1,j+1] \simeq [i,j,j+1]$$
holds.
\end{proposition}

\begin{example}
\label{example_case_n_4}
Let $Q = [u_1,u_2,u_3,u_4] \subset \mathbb{S}^2$ be a spherical polygon whose vertices are not entirely contained on one hemisphere. By Proposition \ref{proposition_characterization}, this is equivalent to
$$ [1,2,3] \simeq [1,3,4] \not \simeq [1,2,4] \simeq [2,3,4],$$
i.e., the cyclic sequence $\epsilon_1 = [1,2,3]$, $\epsilon_2 = [2,3,4]$, $\epsilon_3 = [3,4,1] = [1,3,4]$ and $\epsilon_4 = [4,1,2] = [1,2,4]$ has 4 sign changes. As we saw earlier, this is equivalent to the polygon $Q$ having 4 (spherical) inflections.

Moreover, $Q$ does not have self-intersections. For if it had (which would be between edges $\overrightarrow{u_1 u_2}$ and $\overrightarrow{u_3 u_4}$), then Proposition \ref{self_intersection_determinants} would imply
$$ [1,2,3] \simeq [2,3,4] \not \simeq [1,2,4] \simeq [1,3,4],$$
contradicting the previous determinant relations.

Our conclusion is that, for a spherical polygon $Q$ with 4 points, not only the condition (a) of Proposition \ref{proposition_characterization} implies the existence of 4 spherical inflections, but also the converse. Besides that, any of these two statements imply that $Q$ does not have self-intersections.
\end{example}

\begin{remark}
Our extra assumption on the spherical polygons not having three points in the same spherical line might seem redundant, since we assume the original polygon $P$ in $\mathbb{R}^3$ to be generic: if its tangent indicatrix $Q$ had three consecutive points $u_i$, $u_{i+1}$ and $u_{n+2}$ in the same spherical line, then $e_i$, $e_{i+1}$ and $e_{i+2}$ would be in the same plane, i.e., the vertices $v_i$, $v_{i+1}$, $v_{i+2}$ and $v_{i+3}$ would be in the same plane.

Notice, however, that if the tangent indicatrix had three non-consecutive points in the same spherical line, say $u_i$, $u_{i+1}$ and $u_j$, that would only mean that $v_j$ and $v_{j+1}$ are in a plane parallel to the plane generated by $v_i$, $v_{i+1}$ and $v_{i+2}$. This does not contradict the genericity of $P$.

The justification of why we can assume $Q$ to have this extra property rests on the following remark: given a spherical polygon without $Q \subset \mathbb{S}^2$, we can perturb its vertices slightly so that $Q$ will not have three (spherically) collinear vertices, but preserving at the same time not only the property of not being entirely contained in a hemisphere but also the property of not having self-intersections. Moreover, if $Q$ does not have three consecutive collinear vertices (which is the case), this perturbation can be done without altering the state of a triple of vertices of $P$ of being a flattening or not. 
\end{remark}

\section{Good vertices and proof of the Main Result}

Theorem \ref{segre_theorem_discrete} follows from the following theorem:

\begin{theorem}
\label{segre_theorem_discrete_spherical}
Let $Q=[u_1,...,u_n] \in \mathbb{S}^2$ ($n\geq 4$) be a spherical polygon in balanced position and without self-intersections. Then $Q$ has at least four spherical inflections.
\end{theorem}

The proof will need some lemmas. Given a spherical polygon $Q$ and any of its vertices $u_i$, denote by $Q-u_i$ the polygon $[u_1,...,\hat{u_i},...,u_n]$, obtained from $Q$ by deleting the vertex $u_i$ along with the edges $[u_{i-1},u_i]$ and $[u_i,u_{i+1}]$, and adding the edge $[u_{i-1},u_{i+1}]$ to connect vertices $u_{i-1}$ and $u_{i+1}$.

\begin{definition}
A spherical polygon $Q = [u_1,...,u_n]$ is \textbf{simple} if it does not have self-intersections. A vertex $u_i$ is said to be \textbf{good} if the spherical polygon $Q - u_i$ is simple. Otherwise $u_i$ is said to be \textbf{bad}.
\end{definition}

\begin{lemma}
\label{lemma_resulting_polygon}
Let $Q = [u_1,u_2,...,u_n]$ be a balanced, simple spherical polygon, ($n\geq 4$). Then the set
$$ Y = \{u_i \in Q; u_i \text{ is good}\}$$
has at least four elements.
\end{lemma}
\begin{proof}
Since $Q$ is simple, it divides the sphere $\mathbb{S}^2$ into two disjoint, open regions $R_1$ and $R_2$. The fact that $Q$ is balanced implies, by Lemma \ref{lemma_all_R3}, that for any point $u$ of $\mathbb{S}^2$, $u$ can be expressed as a non-negative combination of vertices of $Q$ (at most three of them, by Proposition \ref{proposition_caratheodory}), i.e., $u$ in the inside of the triangle spanned by vertices $u_i$, $u_j$ and $u_k$ of $Q$.

Therefore, the sphere can be subdivided into triangles whose vertices are the vertices of $Q$. Choose any such triangulation $T$ of the sphere whose triangles are entirely contained either in $R_1 \cup Q$ or $R_2 \cup Q$. Such triangulations always exist in this case (see figure \ref{figure_triangulations}). (For instance, for vertex $u_1$, connect to it all other vertices $u_i$ such that the  spherical segment $[u_1,u_i]$ (i.e., the segment that minimizes distance between the points) only intersects $Q$ at $u_1$ and $u_i$; then connect $u_2$ to all other vertices $u_i$ such that the spherical segment $[u_2,u_i]$ does not intersect $Q$ and the previous added segments, except of course at the vertices of $Q$; and so on.) Notice that, since $n \geq 4$, this triangulation has at least 4 triangles.

For the triangulation $T$ restricted to the region $R_1$ (denoted by $T_1)$, consider its dual graph $G_1$ (a triangle $\triangle_1 \in T$ is considered a vertex and is connected to another triangle $\triangle_2$ if both have a common edge which is not in $Q$). Since the triangulation only uses triangles with vertices in $Q$, then the dual graph $G_1$ is a \textbf{tree}, i.e., it is connected and does not have cycles. By a basic Theorem in Graph Theory, such a graph (provided it has at least two vertices, which is the case), has at least two \textbf{leaves}, i.e., 2 vertices adjacent to only one other vertex (see figure \ref{figure_triangulations}).

In terms of the triangulation $T_1$, this means that there are two triangles $\triangle_1$ and $\triangle_2$ in $T_1$ with only one edge in the relative interior of the region $R_1$. For $\triangle_1$, let $u_i$ be the vertex adjacent to the edges of $\triangle_1$ that are contained in $Q$. Since the edge $[u_{i-1},u_{i+1}]$ of $\triangle_1$ is entirely contained in $R_1$, this means in particular that it does not intersect $Q$ at any other edge, i.e., $Q - u_i$ is simple. In other words, $u_i$ is good. By the same argument applied to $\triangle_2$, we obtain another good vertex $u_j$.

Proceeding analogously to the triangulation $T$ restricted to the region $R_2$, we obtain other two good vertices.
\end{proof}

\begin{figure}
    \centering
    \includegraphics[scale=0.2]{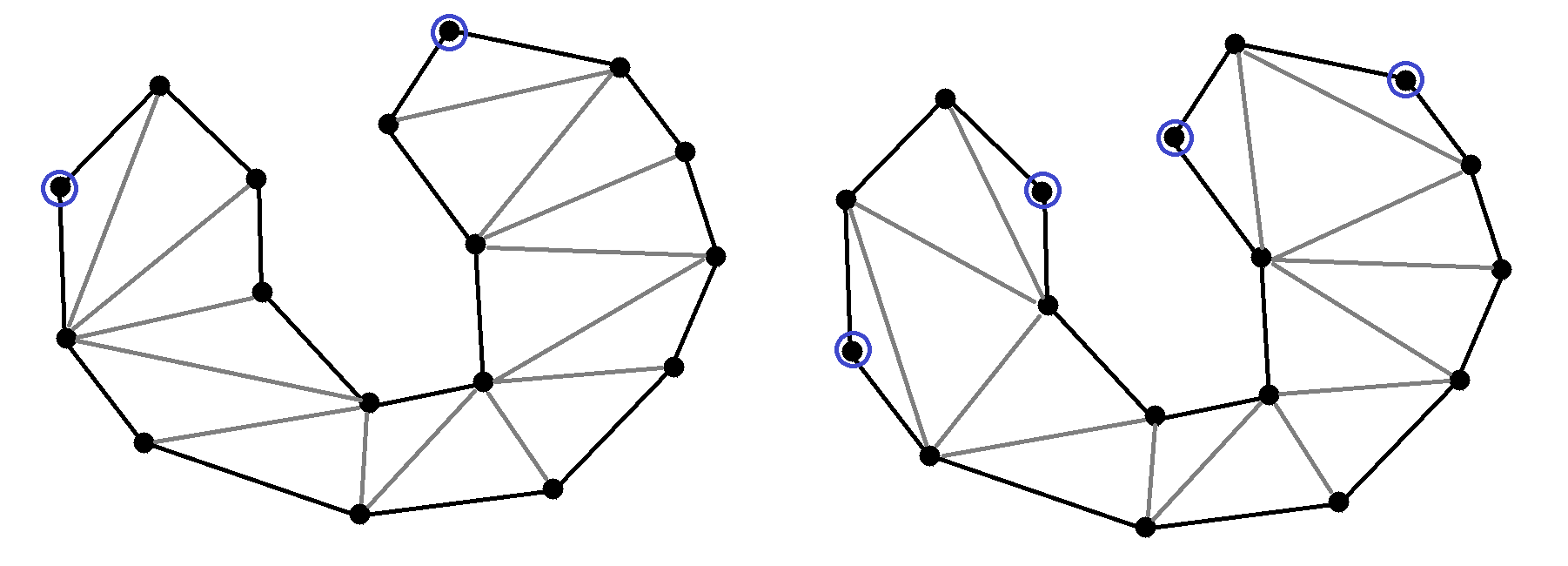}
    \caption{Two possible triangulations for a region determined by a spherical polygon (again, we represent such objects on the plane in order to aid visualization). Notice that different triangulations might lead to different sets of good vertices. Our argument, however, guarantees that for any triangulation there will always be at least 2 such vertices per region.}
    \label{figure_triangulations}
\end{figure}

Figure \ref{figure_balanced_hypothesis_necessary} shows that the balanced position hypothesis on the spherical polygon is necessary, even for a large number of vertices.

\begin{figure}
    \centering
    \includegraphics[scale=0.25]{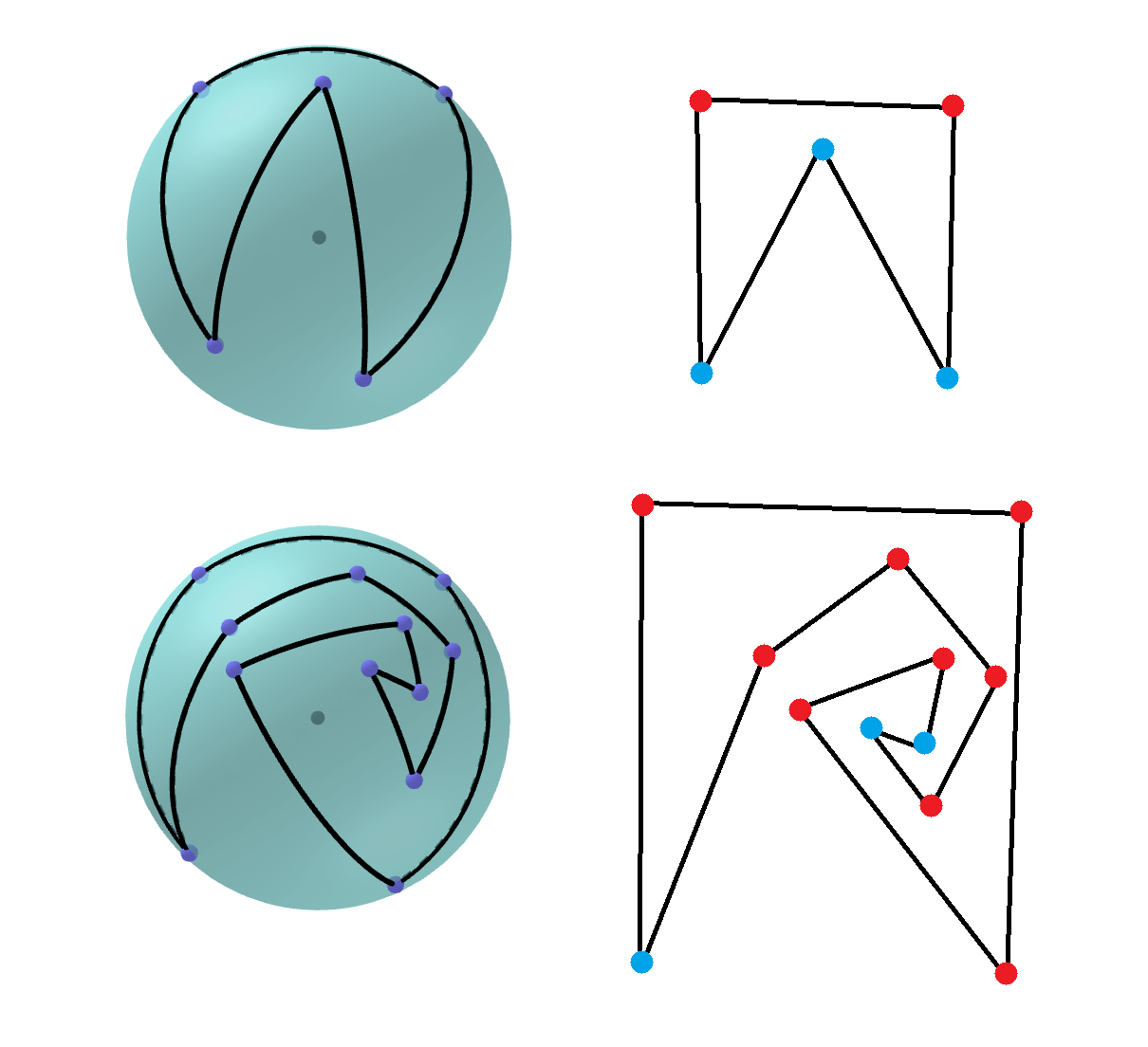}
    \caption{Two examples showing that the balanced position hypothesis is necessary for Lemma \ref{lemma_resulting_polygon} to be true. On the left the spherical polygons, on the right their planar version to aid visualization. The vertices in blue are good, while the ones in red are bad.}
    \label{figure_balanced_hypothesis_necessary}
\end{figure}

\begin{lemma}
\label{lemma_special_vertex}
Given a balanced, simple spherical polygon $Q=[u_1,u_2,...,u_n]$, $n\geq 5$, there is at least one good, nonessential vertex $u_i$.
\end{lemma}
\begin{proof}
By Lemma \ref{lemma_pick_up}, the set
$$ X = \{u_i \in Q; u_i \text{ is nonessential} \}$$
has at least $n-3$ elements. By Lemma \ref{lemma_resulting_polygon}, the set
$$ Y = \{u_i \in Q; u_i \text{ is good}\} $$
has at least four elements. Therefore, the set $X \cap Y$ has at least one element, i.e., there is at least one good, nonessential vertex.
\end{proof}

\begin{lemma}
\label{lemma_inflections_not_increase}
Given a simple spherical polygon $Q$, let $u_i$ be a good vertex of $Q$. Then the number of spherical inflections of $Q$ is greater or equal to the number of spherical inflections of the resulting spherical polygon $Q-u_i$.
\end{lemma}
\begin{proof}
Given $u_i \in Q$, the polygon $Q - u_i$ will be formed by deleting $u_i$ along with the (spherical) edges $[u_{i-1},u_i]$ and $[u_i,u_{i+1}]$ from $Q$, and by adding the edge $[u_{i-1},u_{i+1}]$. Figure \ref{figure_possible_impossible} depicts two of the many possibilities (we represent them on the plane instead of the sphere to aid visualization).

If $u_i$ is the vertex of the conclusion of Lemma \ref{lemma_special_vertex}, then the situation of figure \ref{figure_possible_impossible} (b) cannot happen: if at least one of the vertices $u_{i-2}$ and $u_{i+2}$ were in the inside of the spherical triangle formed by the vertices $u_{i-1}$, $u_i$ and $u_{i+1}$, then $Q$ would either have a self-intersection (which is impossible by hypothesis) or $Q - u_i$ would have a self-intersection (which is not true due to the choice of $u_i$).

Therefore, all possible possibilities are, up to symmetry, the ones represented in figures \ref{figure_three_possible_simple_cases} and \ref{figure_seven_possible_not_simple_cases} (again, we represent these configurations on the plane instead of the sphere). Denoting by $d_i(x)$ the number of spherical inflections of $Q$ minus the number of spherical inflections of $Q-u_i$ in configuration $(x)$, we see that $d_i(a)=0$, $d_i(b)=0$, $d_i(c)=+2$, $d_i(d)=0$, $d_i(e)=+2$, $d_i(f)=0$, $d_i(g)=+2$, $d_i(h)=0$, $d_i(i)=+2$ and $d_i(j)=+4$. Since all these numbers are either positive or zero, the lemma is proved.
\end{proof}

\begin{figure}
    \centering
    \includegraphics[scale=0.4]{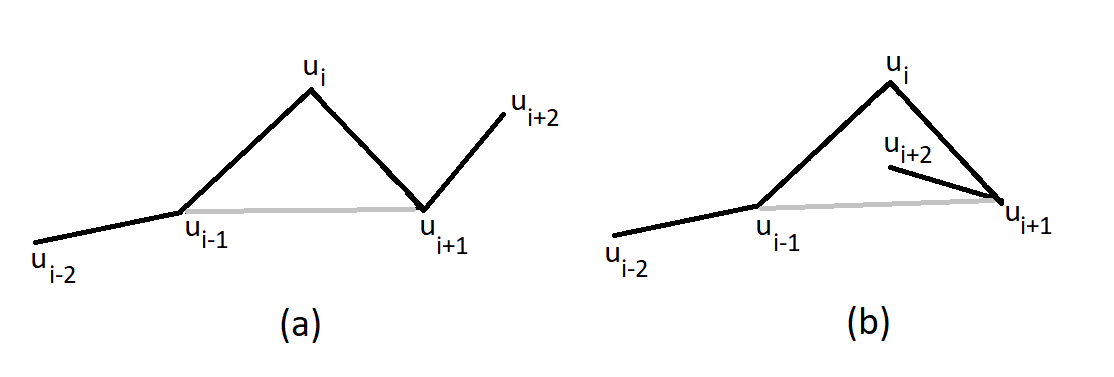}
    \caption{Two of many possibilities}
    \label{figure_possible_impossible}
\end{figure}

\begin{figure}
    \centering
    \includegraphics[scale=0.4]{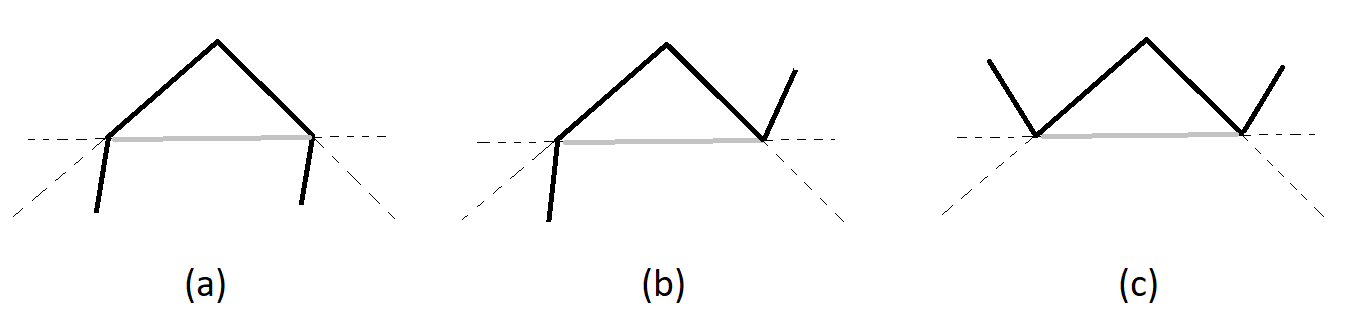}
    \caption{Three possible simple cases}
    \label{figure_three_possible_simple_cases}
\end{figure}

\begin{figure}
    \centering
    \includegraphics[scale=0.35]{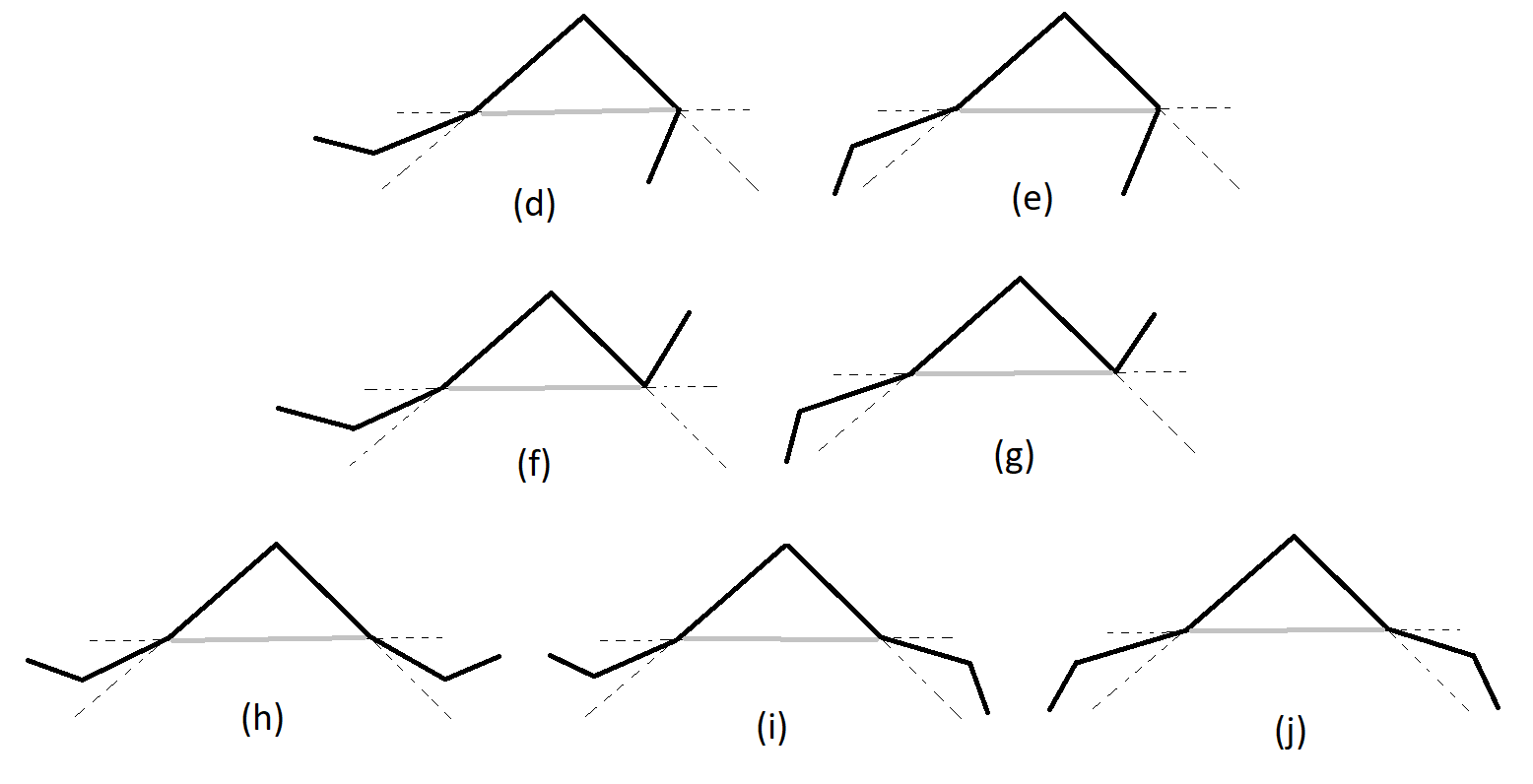}
    \caption{Seven possible cases in which at least one of the adjacent edges might change its condition of being an inflection or not.}
    \label{figure_seven_possible_not_simple_cases}
\end{figure}

\begin{proof}
(of Theorem \ref{segre_theorem_discrete_spherical}) The proof in on induction on the number of vertices of $Q$. The case $n=4$ is Example \ref{example_case_n_4}, for which the result is valid.

Assume that the result holds for spherical polygons with $n$ points. Suppose we are given a spherical polygon $Q$ with $n+1$ points.

By Lemma \ref{lemma_special_vertex}, there is at least one point $u_i$ so that the resulting polygon $Q - u_i$ is balanced and simple. By Lemma \ref{lemma_inflections_not_increase}, the number of inflections of $Q$ is greater or equal to the number of inflections of $Q-u_i$. By the induction hypothesis, however, the number of inflections of $Q-u_i$ is greater or equal to four.
\end{proof}

\section{Applications to Spherical Polygons}

 Besides the original Segre's Theorem for spherical curves, there are in the literature other interesting results regarding smooth curves. Among these results we have the \textbf{Tennis Ball Theorem} and a theorem by Möbius on smooth projective curves (which can be formulated through the notion of spherical \textit{centrally symmetric curves}). In \cite{Ovsienko_article} and \cite{Ovsienko_book}, Ovsienko and Tabachnikov state discrete analogs of these theorems as Conjectures, adding that it would be interesting to find discrete proofs of these results.
 
 Since both of these theorems follow from Segre's Theorem, while our proof of the latter result is entirely discrete, our approach follows the outline set by Ovsienko and Tabachnikov. Before stating and proving these results, we need a preliminary remark.

\begin{remark}
\label{planar_is_always_balanced}
We assume the following \textbf{convention}: a simple spherical polygon $Q$ which is contained in a spherical line will be considered \textbf{balanced} and all its edges will be considered \textbf{spherical inflections}. Note that Definition \ref{definition_balanced} does not apply here since we assumed then that the points of $Q$ would not be in the same spherical line.

The reason for this convention is that such a spherical polygon can always be realized as a tangent indicatrix of a planar polygon $P$ in $\mathbb{R}^3$. Definition \ref{definition_balanced} could be phrased in terms of \textbf{open} hemispheres instead of \textbf{closed} ones in order to contain the planar case, but the proofs involving this alternative notion would always require some argument of perturbation of hemispheres.

Moreover, since the notion of inflection we use is related to the change of signs of the cyclic sequence of determinants, it is a way to mimic the smooth idea of the torsion going from negative to positive (or vice-versa), i.e., passing through zero. For a planar spherical polygon, all determinants $[u_i,u_{i+1},u_{i+2}]$ are zero, hence it is reasonable to consider all edges as inflections.
\end{remark}

As a first application of Theorem \ref{segre_theorem_discrete_spherical} we have the following result:

\begin{theorem}
\label{discrete_tennis_ball_theorem}
(Discrete Tennis Ball Theorem) If a spherical, simple polygon $Q = [u_1,...,u_n]$ ($n \geq 4$) divides the sphere into two regions with the same area, then $Q$ has at least 4 spherical inflections.
\end{theorem}
\begin{proof}
If $Q$ is a spherical line, then the result follows by Remark \ref{planar_is_always_balanced}.

Suppose now that $Q$ is not a spherical line. Since $Q$ is simple, it suffices by Theorem \ref{segre_theorem_discrete_spherical} to show that $Q$ is balanced. If it were not balanced, then $Q$ would be contained in a closed hemisphere $H$. Hence one of the two regions $R_1$ and $R_2$ determined by $Q$ would be contained in $H$ (say $R_1 \subset H$). Since $Q$ is not planar, $R_1 \neq H$ and therefore $area(R_1) < area(H) = 2 \pi$, contrary to hypothesis that $area(R_1) = area(R_2) = 2 \pi$.
\end{proof}

Before our second application of Theorem \ref{segre_theorem_discrete_spherical}, we need a definition:

\begin{definition}
For a set $X \subset \mathbb{R}^d$, define $-X$ as $-X = \{-x; x \in X\}$. We say that $X$ is \textbf{centrally symmetric} if $-X = X$.
\end{definition}

\begin{proposition}
\label{proposition_centrally_symmetric}
Let $Q$ be a simple, centrally symmetric spherical polygon. Then

(a) $Q$ is balanced.

(b) $Q$ divides the sphere into two regions with the same area.
\end{proposition}
\begin{proof}
(a) We may assume $Q$ is not a spherical line. If $Q$ were not balanced, then it would be contained in a closed hemisphere $H$. Hence $-Q \subset -H$. Since $Q$ is centrally symmetric, $Q = -Q \subset -H$, which implies that $Q \subset H \cap -H$, i.e., $H$ is a spherical line, contrary to our assumption.

(b) Let $R_1$ and $R_2$ be the two (connected) regions of $\mathbb{S}^2$ determined by $Q$. Since both $\mathbb{S}^2$ and $Q$ are centrally symmetric and $R_1$ and $R_2$ are connected, we have that $-R_1 = R_2$ and $-R_2 = R_1$. Since the operation $-X$ on sets preserves area, the result follows.
\end{proof}

The following result is a discrete analog of a theorem by Möbius. Recall that the indices of the vertices are always taken modulo the number of vertices of the polygon.

\begin{theorem}
A simple, centrally symmetric spherical polygon $Q$ with at least $2n$ vertices ($2n \geq 6$) has at least 6 inflections.
\end{theorem}
\begin{proof}
We may assume that $Q$ is not a spherical line. By Proposition \ref{proposition_centrally_symmetric} (a) and Theorem \ref{segre_theorem_discrete_spherical} (or also by Proposition \ref{proposition_centrally_symmetric} (b) and Theorem \ref{discrete_tennis_ball_theorem}), $Q$ has at least 4 inflections. Recall that, in terms of determinants, a pair $\{u_i,u_{i+1}\}$ is an inflection if and only if the determinants $[i-1, i, i+1]$ and $[i, i+1, i+2]$ have opposite signs. From this the following facts follow:

(i) Since $Q$ is centrally symmetric (hence $u_{i+n} = -u_i$), the pair $\{u_i,u_{i+1}\}$ is an inflection if and only if $\{u_{i+n},u_{i+n+1}\}$ is an inflection, because in both cases there will be a sign change of determinants.

(ii) Moreover, if the sign change in $\{u_i,u_{i+1}\}$ was from negative to positive (resp. from positive to negative), then the sign change in $\{u_{i+n},u_{i+n+1}\}$ will be from positive to negative (resp. from negative to positive), by the same reason in (i).

If the inflections already obtained are 
$$ \{u_i,u_{i+1}\}, \{u_j,u_{j+1}\}, \{u_k,u_{k+1}\} \text{ and } \{u_l,u_{l+1}\},$$
then by fact (i) the pairs
$$ \{u_{i+n},u_{i+n+1}\}, \{u_{j+n},u_{j+n+1}\}, \{u_{k+n},u_{k+n+1}\} \text{ and } \{u_{l+n},u_{l+n+1}\}$$
are also inflections. There might be some repetitions if some of the first 4 inflections are symmetric to each other. If that does not happen, then we obtain in total 8 inflections. If there is only one pair of symmetric inflections among these first ones, then we obtain in total 6 inflections. Finally, if there are two pairs of symmetric inflections among the first 4 ones, then we still have only 4 inflections. In this case, we label these inflections simply as
$$ \{u_i,u_{i+1}\}, \{u_j,u_{j+1}\}, \{u_{i+n},u_{i+n+1}\} \text{ and } \{u_{j+n},u_{j+n+1}\},$$
with $i<j<i+n<j+n$. We may assume, without loss of generality, that in inflection $\{u_i,u_{i+1}\}$ the sign change went from positive to negative. Consequently, the sign change in $\{u_{i+n},u_{i+n+1}\}$ goes from negative to positive (by fact (ii)). Since inflection $\{u_j,u_{j+1}\}$ happens between them (hence, changing the sign), there must be an odd extra number of inflections between $\{u_i,u_{i+1}\}$ and $\{u_{i+n},u_{i+n+1}\}$ in order to compensate for the change. In particular, there is at least one other inflection $\{u_k,u_{k+1}\}$, with $i < k < i+n$ and $k \neq j$. By fact (i) again, edge $\{u_{k+n},u_{k+n+1}\}$ is also an inflection (a new one). We have, thus, proved that also in this case $Q$ has at least 6 inflections.
\end{proof}

\begin{corollary}
    A a space polygon $P$ with at least $2n$ vertices ($2n \geq 6$) and whose tangent indicatrix is simple and centrally symmetric must have at least 6 flattenings.
\end{corollary}

\section*{Acknowledgments}

The first author has been supported by CAPES postgraduate grants. The second author wants to thank CNPq and CAPES (Finance Code 001) for financial support during the preparation of this manuscript.

\printbibliography

\end{document}